\documentclass[12pt,reqno]{amsart}

\usepackage{amssymb,amsmath}
\usepackage{a4wide}
\usepackage{xcolor}

\newtheorem{thm}{Theorem}[section]
\newtheorem{pro}[thm]{Proposition}

\newtheorem{cor}[thm]{Corollary}
   
\newtheorem{lem}[thm]{Lemma}

\newcommand{\cD}{{\mathcal{D}}}

\newcommand\CC{\mathbb{C}}

\newcommand\TT{\mathbb{T}}
\newcommand\DD{\mathbb{D}}

\DeclareMathOperator{\supp}{supp}

\DeclareMathOperator{\dist}{dist}

\begin{document}

\title[Havin-Mazya uniqueness Theorem for Dirichlet spaces ]{Havin-Mazya type uniqueness Theorem for Dirichlet spaces}

\author[H.Bahajji]{H. Bahajji-El Idrissi}
\address{Laboratoire Analyse et Applications URAC/03. B.P. 1014 Rabat, Morocco}
\email{hafid.fsr@gmail.com}

\author[O.El-Fallah]{O. El-Fallah}
\address{Laboratoire Analyse et Applications URAC/03,  Mohammed
V University in Rabat, B.P. 1014 Rabat, Morocco}
\email{elfallah@fsr.ac.ma}

\author[K.Kellay]{K. Kellay}
\address{IMB\\Universite de  Bordeaux \\
351 cours de la Liberation\\33405 Talence \\France}
\email{kkellay@math.u-bordeaux.fr}

\thanks{Research partially supported by "Hassan II Academy of Science and Technology" for the first and the
 second authors. The research of third author is partially supported by the project ANR-18-CE40-0035 and the Joint French-Russian Research Project PRC-CNRS/RFBR 2017-2019. }

\begin{abstract}
Let $\mu$ be a positive finite Borel measure on the unit circle. The associated Dirichlet space $\cD(\mu)$  consists of holomorphic functions on the unit disc whose derivatives are square integrable when weighted against the Poisson integral of $\mu$.  We  give a sufficient condition   on  a Borel subset $E$ of the unit circle which ensures that $E$ is a uniqueness set for 
$\cD(\mu)$. {We also give somes  examples of positive Borel measures $\mu$  and uniqueness sets for $\cD(\mu)$.}
 \end{abstract}
\maketitle

\section{Introduction}

Let $\mu$ be a positive finite Borel measure on the unit circle $\TT$,  the {\it harmonically weighted Dirichlet space}  $\cD^h(\mu)$ is the set of all function{s}  $f\in L^2(\TT)$ for which 
 $$\cD_{\mu}(f)=\int_{\TT}\cD_\xi(f)d\mu(\xi)<\infty,$$
 where $\cD_\xi(f)$ is the local Dirichlet integral of $f$ at $\xi\in \TT$ given by
 $$\cD_{\xi}(f):=\int_{\TT}\Big|\frac{f({\zeta})-f(\xi)}{\zeta-\xi}\Big|^2\frac{\vert d\zeta\vert}{2\pi}.$$
 The space $\mathcal{D}^h(\mu)$ is endowed with the  norm 
$$\|f\|_{\mu}^{2}:=\|f\|^{2}_{L^2(\TT)}+\cD_{\mu}(f).$$
 The analytic weighted Dirichlet space $\cD(\mu)$  is defined { by}   
$$\cD(\mu)=\{f\in  \cD^h(\mu)\text{ : } \widehat{f}(n)=0,  \; n<0\}.$$
Let $\DD$ be the open unit disc in the complex plane. Let $H^2$ denote the Hardy space of analytic functions on $\DD$. { As usual, we use the identification} $H^2=\{f\in L^2(\TT)\text{ : } \widehat{f}(n)=0,  \; n<0\}$.  Since $\cD(\mu)\subset H^2$,
every function $f\in \cD(\mu)$ has non-tangential limits almost everywhere on $\TT$. We denote by $f(\zeta)$ the non-tangential limit of $f$ at $\zeta\in \TT$ if it exists.  Note that by Douglas Formula the Dirichlet-type  space $\cD(\mu)$  is the set of analytic  functions $f\in H^2$, such that
  $$\int_\DD |f'(z)|^2 P_\mu(z) dA(z)<\infty, $$
where $dA(z)=dxdy/\pi$ stands for the normalized area measure in $\DD$ and $P_\mu$ is the Poisson integral of $\mu$ given by
$$P_\mu(z):=\int_\TT\frac{1-|z|^2}{|\zeta-z|^2}d\mu(\zeta),\qquad z\in \DD.$$
For a proof of this {fact } see \cite[Theorem 7.2.5]{EKMR} and for numerous results on the Dirichlet-type space and operators acting thereon
see  \cite{EKMR, EEK1, Ri,Ri1,RS2,RS3}.\\

 { The capacity associated with ${\cD (\mu)}$ is denoted by $c_\mu$ and is given by 
  $$c_\mu(E):=\inf\left\{\|f\|_\mu^2 \text{ : } f\in \cD^h(\mu) \text{, } |f|\geq 1 \text{ a.e. on a neighborhood of } E \right\}.$$
  Since the $L^2$ norm is dominated by  the Dirichlet norm $\|.\|_\mu$, it is obvious that $c_\mu$--capacity 0 implies Lebesgue measure 0.   We say that a  property  holds $c_\mu$-quasi-everywhere ($c_\mu$-q.e.)  if it holds everywhere outside a set of zero $c_\mu$ capacity.  {Note that $c_\mu$-q.e implies a.e. and we have  
 $$c_\mu(E) = \inf\left\{\|f\|_\mu^2 \text{ : } f\in \cD^h(\mu) \text{, } |f|\geq 1 \;\;    c_\mu
 \text{-q.e. on  }E\right\}.$$
 For more details see \cite{G}. }Furthermore every function $f\in \cD(\mu)$ has non-tangential limits $c_\mu$-q.e. on $\TT$ \cite[Theorem 2.1.9]{C}.   { Note also that if $E$
is a closed subset of $\TT$ such that $c_\mu(E)= 0$, then there exists a
function $f\in \cD(\mu)$ uniformly continuous on $\TT$  such that the zero set $\mathcal{Z}(f):=\{\zeta \in \TT \text{ : } f(\zeta)=0\}=E$ \cite[Theorem 1]{EL}.} 

{Recall that if $I\subset \TT$ be the arc of lenght $|I|=1-\rho$ with midpoint $\zeta\in \TT$, then 
\begin{equation}\label{capinterval}
\frac{1}{c_\mu(I)}\asymp 1+\int_{0}^{\rho}\frac{dr}{(1-r)P_\mu(r\zeta)+(1-r)^2},
\end{equation}
where the implied constants are absolute. For the proof see \cite[Theorem 2]{EEK}. }\\
 Let $E$ be a subset of $\TT$, the set $E$ is said to be a uniqueness set for $\cD(\mu)$ if, for each $f\in \cD(\mu)$ such that its non-tangential limit $f= 0$ $c_\mu$--q.e on $E$, we have $f=0$. \\
{ It is well known that for the Hardy space the uniqueness sets coincide with the sets of positive length on $\TT$. } Note that if $d\mu= dm$ the normalized arc measure on $ \TT$, then the space $\cD(\mu)$  coincides with the classical Dirichlet space $\cD$ given by
$$
\cD = \{ f \in H^2:\ \displaystyle \int _{\DD}|f'(z)|^2 dA(z)<\infty \}.
$$
In this case, $c_m$ is comparable to the logarithmic capacity and $c_m (I) \asymp  |\log |I||^{-1}$ for every arc $I\subset \TT$,  \cite[Theorem 14]{Mey}. \\
 Khavin and Maz'ya proved in \cite{KM} that a  Borel subset  $E$ of $\TT$ is a uniqueness set for $\cD$ if there exists  a family  of pairwise disjoint open arcs  $(I_n)$  such that 
 $$
\sum_{n}|I_n|\log \frac{|I_n|}{c_{m}(E\cap  I_n)} =-\infty.
$$
For other uniqueness results for $\cD$ see also \cite{C2,C1,G1,EKMR,K}.

Our aim in this paper is to extend  Khavin and Maz'ya uniqueness theorem to  general Dirichlet spaces $\cD(\mu)$.

 Let $\gamma >1$ and let $I= (e^{ia}, e^{ib})$. The arc $\gamma I$ is given by 
$$\gamma I= (e^{i(a - (\gamma-1)\frac{b-a}{2} )},e^{i(b + (\gamma-1)\frac{b-a}{2}) }).$$
The main result of this paper is the following theorem.
 \begin{thm}\label{unicite}  Let  $E$ be a Borel subset of  $\TT$. Suppose that there exists a family  of  open arcs  $(I_n)$ such that  $(\gamma I_n)$ are pairwise disjoint for some $\gamma>1$ and 
 $$
\sum_{n}|I_n|\log \frac{|I_n|}{c_{\mu}(E\cap  I_n)} =-\infty; 
$$
then $E$ is a uniqueness set for $\cD(\mu)$. 
\end{thm}

 { The key of the proof of the this theorem is an upper estimate of the average 
  $$\frac{1}{|I|} \int_{I} |f(\zeta)||d\zeta|, $$ 
  for $f\in \cD(\mu)$ vanishing on a set $E\subset \TT$, in terms of capacity  of $E\cap  I$ , for any open arc $I$.\\
 
The next section is devoted to the proof of  Theorem \ref{unicite}. In section 3 we give some examples of positive Borel measures $\mu$  and uniqueness sets for $\cD(\mu)$. }\\

 Throughout the paper, we use the following notations:  $A\lesssim   B$  means that there is an absolute constant $C$ such that $A \leq C B$ and 
 $A \asymp B$ if both $A\lesssim   B$ and $B\lesssim   A$.

  \section{Proof}
  { First, let us introduce some notations which will be useful in the sequel.
 Let $J$ and $L$ be 
 arcs of $\TT$ and let $f$ be a Borel  function defined on $\TT$. We set 
 $$\cD_{J,L,\mu}(f):=\int_{\zeta\in J}\int_{\xi\in L}\frac{|f(\zeta)-f(\xi)|^2}{|\zeta-\xi|^{2}}\frac{|d\zeta|}{2\pi} d\mu(\xi),$$
 and 
 $$\langle f\rangle_J:=\frac{1}{|J|}\int_J |f(\xi)||d\xi|.$$
The following lemma is the key in the proof of theorem \ref {unicite}.
\begin{lem}\label{capacitepoincare} Let $\gamma>1$
and  let $f\in\mathcal{D}(\mu)$  such that  $f|_{E}=0$ for some Borel subset $E$ of $\TT$. Then, for any open arc  $I\subset \TT$ 
$$\langle f\rangle_{ I}^2\leq \kappa\frac{\cD_{\gamma I,\TT,\mu}(f)+\int_{\gamma I}|f|^2}{c_{\mu}(E\cap  I)}, $$
where $\kappa$ is a constant depending only on $\gamma$.
\end{lem}
}
\begin{proof} 
  { Without loss  of generality, we can suppose that $I=(e^{-i \theta}, e^{i\theta})$ with $2\theta<\pi$.  In this case 
$\gamma I = (e^{-i\gamma \theta},e^{i\gamma \theta})$.}
 
 Let $\phi$ be a positive function  on  $\TT$, $0\leq \phi\leq 1$, such that  $\displaystyle \supp \phi= \frac{1+\gamma}{2}I$, $\phi=1$ on $ I$ and  
 $$\displaystyle |\phi(\zeta)-\phi(\xi)|\leq\frac{c_1}{|I|}|\zeta-\xi|, \qquad \zeta,\xi\in \TT.$$
where $c_1$ depends  only on $\gamma$.  Set  $\displaystyle L= \frac{1+\gamma}{2} I$ and  consider 
$$F(\zeta):= \phi(\zeta)\Big|1-\frac{|{f}(\zeta)|}{\langle f\rangle_{L}}\Big|,\qquad \zeta\in \TT.$$

Hence  $F\geq 0$ and $F=1$ $c_\mu$-q.e on $E\cap I$. Therefore, 
$$
c_{\mu}(E\cap  I)
\leq \|F\|_{\mu}^2.
$$

 We claim that 
 \begin{equation}\label{normfinal}
\|F\|_{\mu}^2\leq \kappa  \frac{ \cD_{\gamma I,\TT,\mu}({f})+\int_{L}|f|^2}{\langle f\rangle_{L}^{2}}.
\end{equation}
where $\kappa$ depends  only on $\gamma$.

Indeed, we have
 
\begin{eqnarray}\label{inegaliteF}
\|F\|_{\mu}^2&=&\int_\TT|F(\zeta)|^2\frac{|d\zeta|}{2\pi}+ \int_{\TT}\int_{\TT} \frac{|F(\zeta)-F(\xi)|^2}{|\zeta-\xi|^{2}}\frac{|d\zeta|}{2\pi}d\mu(\xi)\nonumber\\
&\leq&  \frac{1}{\langle f\rangle_{L}^{2}}\int_{{L}}{|\langle f\rangle_{L}-|{f}(\zeta)||^2}\frac{|d\zeta|}{2\pi}+ \int_{\gamma I}\int_{\gamma I} \frac{|F(\zeta)-F(\xi)|^2}{|\zeta-\xi|^{2}}\frac{|d\zeta|}{2\pi}d\mu(\xi)\nonumber\\
&&+\frac{1}{\langle f\rangle_{L}^{2}}\int\limits_{\zeta\in\TT\backslash {\gamma I}}\int\limits_{\xi\in L }\frac{|\langle f\rangle_{L}-|{f}(\xi)||^2}{|\zeta-\xi|^{2}}\frac{|d\zeta|}{2\pi}d\mu(\xi)
+\nonumber\\
&&
\frac{1}{\langle f\rangle_{L}^2}\int\limits\limits_{\zeta\in L }\int\limits_{\xi\in\TT\backslash {\gamma I}}\frac{|\langle f\rangle_{L}-|{f}(\zeta)||^2}{|\zeta-\xi|^{2}}\frac{|d\zeta|}{2\pi}d\mu(\xi)\nonumber\\
&=&\frac{A}{2\pi \langle f\rangle_{L}^2}+\frac{B}{2\pi}+\frac{C}{2\pi\langle f\rangle_{L}^2}+\frac{D}{2\pi\langle f\rangle_{L}^2}.
\end{eqnarray}

By Cauchy-Schwarz inequality we have

\begin{equation}\label{CS}
|\langle f\rangle_{L}-\vert f(\xi)\vert|^2\leq \frac{1}{|L|}\int_{L} |{f}(\eta)-{f}(\xi)|^2|d\eta|.
\end{equation}

Hence 

\begin{eqnarray}\label{inegalite1}
A &:= &\int_{L}{|\langle f\rangle_{L}-|{f}(\zeta)||^2}|d\zeta|\nonumber\\
&\leq &  \frac{1}{|L|}\int_{L}\int_{L} |{f}(\eta)-{f}(\xi)|^2|d\eta||d\xi|\nonumber\\
&\leq & \frac{2}{|L|}\int_{{L}}\int_{L} |{f}(\eta)|^2|d\eta||d\xi|+\frac{2}{|L|}\int_{L}\int_{L} |{f}(\xi)|^2|d\eta||d\xi|\nonumber\\
&=& 4\int_{L}|f(\zeta)|^2|d\zeta|.
\end{eqnarray}

Next we estimate $B$. 
 For {$(\zeta,\xi)\in \TT \times \TT$}, we have 

\begin{multline}
\label{eqK}
|F(\zeta)-F(\xi)|=
\Big|\phi(\zeta)\Big(\Big|1-\frac{|{f}(\zeta)|}{\langle f\rangle_{L}}\Big|-\Big|1-\frac{|{f}(\xi)|}{\langle f\rangle_{L}}\Big|\Big)+
(\phi(\zeta)-\phi(\xi))\Big|1-\frac{|{f}(\xi)|}{\langle f\rangle_{L}}\Big|\Big|\\
\leq \frac{1}{\langle f\rangle_{L}}|{f}(\zeta)-{f}(\xi)|+ \frac{c_1}{\langle f\rangle_{L}}\frac{|\zeta-\xi|}{ |I|}|\langle f\rangle_{L}-|{f}(\xi)||.
\end{multline}

By  \eqref{eqK} and \eqref{CS} we get 

\begin{eqnarray}\label{inegalite2}
B &:=&  \int_{\gamma I}\int_{\gamma I} \frac{|F(\zeta)-F(\xi)|^2}{|\zeta-\xi|^{2}}|d\zeta|d\mu(\xi)\nonumber \\
  &\leq & \frac{2}{\langle f\rangle_{L}^2} \int_{\gamma I}\int_{\gamma I} \frac{|{f}(\zeta)-{f}(\xi)|^2}{|\zeta-\xi|^{2}}|d\zeta|d\mu(\xi) +\nonumber\\
  && \frac{2 \times c_1^2 }{\langle f\rangle_{L}^2|I|^2 \vert L\vert} \int_{\gamma I}\int_{\gamma I}  \int_{L}|{f}(\eta)-{f}(\xi)|^2|d\eta||d\zeta|  d\mu(\xi)\nonumber\\
  &\leq & \frac{2}{\langle f\rangle_{L}^2} \int_{\gamma I}\int_{\gamma I} \frac{|{f}(\zeta)-{f}(\xi)|^2}{|\zeta-\xi|^{2}}|d\zeta|d\mu(\xi) + \nonumber\\
  &&\frac{2 \times c_1^2\times \gamma  }{\langle f\rangle_{L}^2|I| \vert L\vert} \int_{\gamma I}  \int_{L}|{f}(\eta)-{f}(\xi)|^2|d\eta|  d\mu(\xi)\nonumber\\
  &\leq & \frac{2}{\langle f\rangle_{L}^2} \int_{\gamma I}\int_{\gamma I} \frac{|{f}(\zeta)-{f}(\xi)|^2}{|\zeta-\xi|^{2}}|d\zeta|d\mu(\xi) +\nonumber\\
&&   \frac{c_2 }{\langle f\rangle_{L}^2} \int_{\gamma I}\int_{\gamma I} 
  \frac{|{f}(\eta)-{f}(\xi)|^2}{|\eta-\xi|^{2}}|d\eta|d\mu(\xi)\nonumber\\ 
  &\leq & c_3\frac{\cD_{\gamma I,\gamma I,\mu}({f}) }{\langle f\rangle_{L}^2} . 
\end{eqnarray}

where $c_2$ and $c_3$ depend  only on $\gamma$.

Using again \eqref{CS} , we see that 

\begin{eqnarray}\label{inegalite3}
C&:=&\int_{\zeta\in\TT\backslash \gamma I}\int_{\xi\in L }\frac{|\langle f\rangle_{{L}}-|{f}(\xi)||^2}{|\zeta-\xi|^{2}}|d\zeta| d\mu(\xi)\nonumber\\
&\leq &\int_{\zeta\in\TT\backslash \gamma I} \frac{|d\zeta|}{d(\zeta,L)^2}\int_{\xi\in L }|\langle f\rangle_{L}-|{f}(\xi)||^2d\mu(\xi)\nonumber\\
&\leq &\frac{c_4}{| I |}\int_{\xi\in L }|\langle f\rangle_{L}-|{f}(\xi)||^2d\mu(\xi)\nonumber\\
&\leq &\frac{c_5 }{|I|^{2}}\int_{\gamma I}\int_{\gamma I} |{f}(\eta)-{f}(\xi)|^2|d\eta|d\mu(\xi)\nonumber\\
&\leq& c_5 \int_{ \gamma I}\int_{\gamma I}\frac{|{f}(\eta)-{f}(\xi)|^2}{|\eta-\xi|^{2}}|d\eta|d\mu(\xi)\nonumber\\
& = & 2\pi c_5\cD_{\gamma I,\gamma I,\mu}({f}),
\end{eqnarray}
where $c_4$ and $c_5$ depend only on $\gamma$.
{

We have
$$\vert \zeta-\xi \vert\geq \frac{\gamma-1}{4}|I| \quad  \text{and}\ \quad \vert \eta-\xi\vert \leq  \frac{1+2\gamma}{\gamma-1} \vert \zeta-\xi\vert \qquad  (\xi,\zeta,\eta)\in (\TT\backslash \gamma I )\times L \times L.$$
}
Therefore,  by \eqref{CS} we oblain
 
\begin{eqnarray}\label{inegalite4}
D&:=&\int_{\xi\in\TT\backslash \gamma I}\int_{\zeta\in L }\frac{|\langle f\rangle_{L}-|{f}(\zeta)||^2}{|\zeta-\xi|^{2}}|d\zeta| d\mu(\xi)\nonumber\\
&\leq &\int_{\xi\in\TT\backslash \gamma I}\int_{\zeta\in L }\frac{1}{ |L|} \int_{\eta\in L}\frac{|f(\eta)-f(\xi)+f(\xi)-{f}(\zeta)|^2}{|\zeta-\xi|^{2}}|d\zeta| |d\eta| d\mu(\xi)\nonumber\\
&\leq &2\int_{\xi\in\TT\backslash \gamma I} \int_{\eta\in L}\int_{\zeta\in L }
\frac{ |f(\eta)-f(\xi)|^2}{|\zeta-\xi|^{2}} |d\zeta| |d\eta| d\mu(\xi) +\nonumber\\&&2\int_{\xi\in\TT\backslash \gamma I}\int_{\zeta\in L} \frac{|f(\xi)-{f}(\zeta)|^2}{|\zeta-\xi|^{2}}|d\zeta|  d\mu(\xi)\nonumber\\
&\leq &\frac{c_6}{ |I|}\int_{\xi\in\TT\backslash \gamma I} \int_{\eta\in L}\int_{\zeta\in L }
\frac{ |f(\eta)-f(\xi)|^2}{|\eta-\xi|^{2}} |d\zeta| |d\eta| d\mu(\xi) +\nonumber\\&&2\int_{\xi\in\TT\backslash \gamma I}\int_{\zeta\in L } \frac{|f(\xi)-{f}(\zeta)|^2}{|\zeta-\xi|^{2}}|d\zeta|  d\mu(\xi)\nonumber\\
&= &c_7\int_{\xi\in\TT\backslash \gamma I}\int_{\zeta\in L } \frac{|f(\xi)-{f}(\zeta)|^2}{|\zeta-\xi|^{2}}|d\zeta|  d\mu(\xi)\nonumber\\
&\leq & 2\pi c_7 \cD_{\gamma I,\TT\setminus\gamma I,\mu}({f}),
\end{eqnarray}

where $c_6$ and $c_7$ depend only on $\gamma$.

Combining  \eqref{inegaliteF},   \eqref{inegalite1}, \eqref{inegalite2}, \eqref{inegalite3} and \eqref{inegalite4} we get \eqref{normfinal} 
and the proof is complete.
\end{proof}
\subsection*{Proof of Theorem \ref{unicite}}
Let  $f\in \cD(\mu)$ such that  $f|_E=0$ and  set $\ell=\sum_n |I_n|$. By Lemma \ref{capacitepoincare} and Jensen's inequality it follows that
{
\begin{eqnarray*}
 \int_{\bigcup  I_n}\log|f(\xi)||d\xi|&=&\sum_n |I_n|\frac{1}{ |I_n|}\int_{ I_n}\log|f(\xi)||d\xi|\\
 &\leq&\frac{1}{2}\sum_n  |I_n| \log\langle f\rangle_{ I_n}^{2}\\
 &\leq&\frac{1 }{2} \sum_n |I_n|\log\Big(\kappa \frac{ \cD_{\gamma I_{n},\TT,\mu}({f}) +
 \int_{\gamma I_n}|f(\zeta)|^2|d\zeta|}{c_{\mu}(E\cap  I_n)} 
 \Big)\\
 &=&\frac{1}{2}( \ell\log\kappa +\mathcal{I}),
  \end{eqnarray*}
where 
$$\mathcal{I}= \sum_n |I_n| \log \frac{|I_n|}{c_{\mu}(E\cap  I_n)} +\underbrace{ \sum_n {|I_n|} \log \Big( \frac{\cD_{\gamma I_n,\TT, \mu}({f})+\int_{\gamma I_n}|f(\zeta)|^2|d\zeta|}{|I_n|}\Big)}_{\mathcal{J}}
$$
Using again   Jensen inequality and since $(\gamma I_n)$ are  pairwise disjoint, we get 
 \begin{eqnarray*}
\mathcal{J}&=& \ell \sum_n \frac{|I_n|}{\ell} \log \Big( \frac{\cD_{\gamma I_n,\TT, \mu}({f})+\int_{\gamma I_n}|f(\zeta)|^2|d\zeta|}{|I_n|}\Big)\\
&\leq&\ell \log\Big[\frac{1 }{ \ell}\Big(\sum_n\cD_{\gamma I_{n},\TT,\mu}({f})+\int_{\gamma I_n}|f(\zeta)|^2|d\zeta| \Big)\Big]\\
&\leq &\ell \log \Big[\frac{1}{ \ell}\big(\cD_\mu(f)+\|f\|^2_2 \big)\Big]\\
&=&\ell \log \frac{\|f\|^{2}_{\mu}}{\ell}.
\end{eqnarray*}
Therefore $\mathcal{I}=-\infty$. }So by Fatou's Theorem we obtain $f=0$ and the proof is complete.

\section{Remarks and examples}
\subsection{}  By  \eqref{capinterval}, see also \cite[Theorem 2]{EEK}, we have  $c_\mu ({\zeta})>0$, for some $\zeta \in \TT$, if and only if 
$$\displaystyle \int _0^1\frac{dr}{(1-r)P_\mu(r\zeta)+(1-r)^2}
<\infty.
$$
In this case, for every $f\in \cD (\mu)$, the non tangential limit, $f(\zeta )$, at $\zeta$ exists. We have the following upper estimate
  \begin{lem}\label{R}
  Let $\mu $ be a positive Borel measure on $\TT$ and let $\zeta \in \TT$ such that   $c_\mu({\zeta })>0$. For $\beta \in (0,1/2)$ and  $f\in \cD _\mu$, we have 
  $$
  |f(z)-f(\zeta)|^2\leq C \left ( \displaystyle \int _{S(\zeta, \beta)}|f'(z)|^2P_\mu(z)dA(z)\right )\left (\displaystyle \int _0^{\beta}\frac{dx}{xP_\mu( (1-x)\zeta )+x^2} \right ),
  $$
  where $S(\zeta , \beta)= \{ w\in \DD:\ 1-|w|<\beta,\ {\rm arg}(w\bar{\zeta})<\beta\}$, $z\in S(\zeta, \beta/4)$ and $C$ is an absolute constant.
  \end{lem}
  \begin{proof}
  One can use the same idea as in \cite[Theorem 2]{EEK}.
  \end{proof}

\subsection{} 
 The capacity $c_\mu$ is comparable with analytic capacity given by 
 $$c_\mu^a(E)=\inf\{\|f\|^2_\mu\text{ : } f\in \cD(\mu) \text{ and } |f|\geq 1 \text{ a.e. on a neighborhood of } E\}.$$
 Indeed we have  $c_\mu(E)\leq c_\mu^a(E)\leq 4c_\mu(E)$, see \cite[Lemma 3.1]{EEK1}. On the other hand 
If  $f\in \cD(\mu)$ then $f\in L^{2}(\TT,d\mu)$ and
 $\|f\|_{L^2(\TT,d\mu)}\leq (1+\mu(\TT)^{1/2})\|f\|_{\mu}$, \cite[Theorem 8.1.2]{EKMR}.   
 Therefore 
$$c_\mu(E)\geq \mu(E)/ (1+\mu(\TT)^{1/2})^{2}.$$
  We obtain the following result
   \begin{cor}
  Let $\gamma >1$
and  let $f\in\mathcal{D}(\mu)$  such that  $f|_{E}=0$ for some Borel subset $E$ of $\TT$. Then, for any open arc  $I\subset \TT$ 
$$\langle f\rangle_{ I}^2\leq \kappa\frac{\cD_{\gamma I,\TT,\mu}(f)+\int_{\gamma I}|f|^2}{\mu(E\cap  I)}, .$$
where $\kappa$ is a constant depending only on $\gamma$.
\end{cor}

\subsection{}First observe that for the Dirac measure $\delta _\zeta$ it is known, (see \cite {EKMR}), that $\cD (\delta _\zeta) = \CC + (z-\zeta)H ^2$. Then a Borel set $E$ is a uniqueness set for $\cD (\delta _\zeta)$ if and only if the $E$ has a positive  Lebesgue measure. This result can be extended to some other discrete measures. For a positive Borel measure $\mu$ we will denote 
by $V_2(\mu ) $ the Newtonian potential given by
$$
V_2(\mu ) (\zeta )= \displaystyle \int _{0}^{2\pi} \frac{d\mu (e^{it})}{|e^{it}-\zeta|^2}.
$$
{D. Guillot showed in \cite[Theorem 2.1]{G1} that if there exists $f\in \cD_\mu$ such that $f=0$ $\mu$ a.e on $\TT$ then 
$$
\displaystyle \int _{\TT}\log V_2(\mu) (\zeta)|d\zeta| < \infty.
$$
{ He also proved that the converse is true for all discrete measures.}} The following result is an immediate consequence of \cite {G1}.
\begin{pro}
Consider a positive sequence $(a_n)_{n\geq 1}$  such that $\displaystyle \sum _{n\geq 1} a_n=1$ and let $\mu = \sum _n a_n \delta _{\zeta _n}$ where $\zeta _n \in \TT$. If 
$$
\displaystyle \int _{\TT}\log V_2(\mu) (\zeta)|d\zeta| < \infty,
$$
then a Borel set $E\subset \TT$  is a uniqueness set for $\cD (\mu) $ if and only if $|E|>0$.
\end{pro}
\begin{proof}
Since $\cD(\mu) \subset H^2$, it is obvious that the condition $|E|>0$ is sufficient. Conversely,  Let $E\subset \TT$ be a Borel set  such that $|E|= 0$,  there exists a function $\varphi \in H^{\infty}\setminus \{0\}$ such that $\varphi $ has non-tangential boundary limits $0$ on $E$.  By \cite{G}, there exists $f\in \cD (\mu)\cap H^{\infty} \setminus \{0\}$ such that $f(\zeta _n)=0, \ (n\geq 1)$. Now it suffices to prove that $\varphi f \in \cD (\mu)$. Indeed,

\begin{eqnarray}
\cD _\mu(\varphi f)&= & \int _\TT  \int _\TT\Big|\frac{\varphi f(\zeta)- \varphi f (\zeta') }{\zeta - \zeta '}\Big|^2|d\zeta | d\mu(\zeta ')\nonumber\\
& =& \displaystyle \int _\TT \displaystyle \int _\TT\Big|\frac{\varphi f(\zeta) }{\zeta - \zeta '}
\Big|^2|d\zeta|d\mu(\zeta ')\nonumber\\
& \leq &  \| \varphi \| ^2_{\infty}\displaystyle \int _\TT \displaystyle \int _\TT\Big|
\frac{ f(\zeta) }{\zeta - \zeta '}\Big|^2 |d\zeta | d\mu(\zeta ')\nonumber\\
& = &  \| \varphi \| ^2_{\infty}\displaystyle \int _\TT \displaystyle \int _\TT\Big|\frac{ f(\zeta)-f(\zeta ') }{\zeta - \zeta '}\Big|^2|d\zeta | d\mu(\zeta ')\nonumber\\
& =&  \| \varphi \| ^2_{\infty}\cD _\mu(f)\nonumber
\end{eqnarray}
which completes the proof.
\end{proof}
 
Note that for a positive Borel measures $\mu = \sum _n a_n \delta _{\zeta _n}$ where $\zeta _n \in \TT$ such that 
\begin{equation}\label {div}
\displaystyle \int _{\TT}\log V_2(\mu) (\zeta)|d\zeta| = \infty,
\end{equation}
we have no complete characterization for uniqueness sets for $\cD (\mu )$. In this case,  by \cite {G1} the countable set $\{\zeta\in \TT\text{ : } \mu(\zeta) \neq 0 \}$ is a uniqueness set for $\cD (\mu)$. The situation is more complicated than in the previous case. In fact, using theorem \ref{unicite}, we will construct  smaller uniqueness set for $\cD (\mu )$. Before stating this result, we give an example of discrete measures $\mu$ satisfying (\ref {div}). 
A closed set $E\subset \TT$ is said to be a Carleson set if
$$\int_\TT\log \dist(\zeta,E)|d\zeta|>-\infty.$$
This condition is  equivalent to  $|E|=0$ and 
$$\sum_n |\ell_n|\log |\ell_n|>-\infty,$$
where $\ell_n$ are the  complementary intervals of $E$.  For more informations on Carleson sets, see e.g. \cite[\S 4.4]{EKMR}.

\begin{pro}
Let $E$ be a countable closed set of the unit circle such that $E$ is not Carleson set. Then there exists a discrete measure $\mu$ on the unit circle such that $ E=\{\zeta\in \TT\text{ : } \mu(\{\zeta\})\neq 0\}$ and (\ref {div}) is satisfied.
\end{pro}
\begin{proof}
Let $\TT \setminus E= \bigcup (a_n,b_n)$ and let 
$$\mu = \displaystyle \sum _{n\geq 1}(b_n-a_n)(\delta _{a_n}+\delta _{b_n}).$$
 We have 
\begin{eqnarray*}
\displaystyle  \int _{\TT}\log V_2(\mu) (\zeta)|d\zeta| & \geq & \displaystyle \sum _{n\geq 1}\displaystyle \int _{a_n}^{b_n}\log \frac{b_n-a_n}{\dist^2(e^{it},\{a_n,b_n \})}dt\\
& \geq & \displaystyle \sum _{n\geq 1}\displaystyle \int _{a_n}^{b_n}\log \frac{b_n-a_n}{(b_n-a_n)^2}dt\\
&=&  \displaystyle \sum _{n\geq 1}(b_n-a_n)\log \frac{1}{b_n-a_n}=+ \infty.
\end{eqnarray*}
\end{proof}
\subsection{} {Here we give  an example of positive measure $\mu$ and  a countable closed set $E$ such that $\mu(E)=0$ and $E$ is uniqueness set for $\cD(\mu)$ (see Corollary \ref{UM0}). This result  was  easily obtain from Theorem \ref{unicite}. We can also obtain this result by more direct method using Lemma \ref{R} and  easy estimates of harmonic measures for some suitably domains.}
\begin{lem} \label {lemcap}
Let $\zeta \in \TT$, $a\in (0,1/2)$. Let $\mu $ be a probability measure on $\TT$ such that $\mu \geq  \sum _{k\geq 2}{a}k^{-2}\delta _{\zeta e^{ia^k}}$, then $c_\mu (\{\zeta \})\gtrsim \sqrt{a}.$
\end{lem}
\begin{proof}
Let 
$$\nu =  \displaystyle \sum _{k\geq 2}\frac{a}{k^2}\delta _{\zeta e^{ia^k}}.$$
 By  \eqref{capinterval}, we have 
$$
 \frac{1}{c_\mu(\{\zeta\})} 
  \lesssim  1+ \displaystyle \int _0^1\frac{dr}{(1-r)P_\nu(r\zeta)+(1-r)^2}.
$$
For $a \leq 1-r$, we have
\begin{eqnarray*}
(1-r)P_\nu(r\zeta)& \asymp  & \displaystyle \sum _{k\geq 2}\frac{a(1-r)^2}{k^2((1-r)^2+a^{2k})}\\
&\asymp&   \displaystyle \sum _{k\geq 2}\frac{a}{k^2}
\asymp   a.
\end{eqnarray*}

For $1-r <a $, let $k_r$ be a real number such that $1-r= a^{k_r}$.  Since $x\longmapsto x^2 a^{2x}$; $x\geqslant 1$, is a decreasing function, we have
\begin{eqnarray*}
(1-r)P_\nu(r\zeta)& \asymp   &\displaystyle \sum _{k\geq 2}\frac{a(1-r)^2}{k^2((1-r)^2+a^{2k})}\\
 &= & \displaystyle \sum _{k\geq k_r}\frac{a}{k^2}+   \displaystyle \sum _{k<k_r}\frac{a(1-r)^2}{k^2a^{2k}}\\
&\asymp&   
\frac{a\log (1/a)}{\log (1/1-r)} .
\end{eqnarray*}
Then we obtain

$$
 \frac{1}{c_\mu(\{\zeta\})} \leq  1+\displaystyle \int _{0}^{1-a}\frac{dr}{ a +(1-r)^2}+ \displaystyle \int _{1-a}^{1}\frac{dr}{ \frac{a\log (1/a)}{\log (1/1-r)} +(1-r)^2}                    
                      \lesssim \frac{1}{\sqrt{a}}
$$

and the lemma is proved. 
\end{proof}
\begin{cor}\label {UM0}
There exists a discrete probability measure $\mu$ on the unit circle and a countable closed set $E \subset \TT$, such that $\mu (E)=0$ and $E$ is a set of uniqueness for $\cD (\mu)$.
\end{cor}
\begin{proof}
Let $E$ be the set define by
$$E:=\left\lbrace \zeta _n = e^{i\theta_{n}}: n\geqslant 3 \right\rbrace \cup \left\lbrace 1 \right\rbrace, $$
with $\theta_{n}:={1}/{\log(n)}$. Let $a_n = \varepsilon (\theta _{n} - \theta _{n+1})$ for some small $\varepsilon >0$. Consider the sequence $\zeta _{n,k}= e^{i\theta _{n,k}}$ where $ \theta_{n,k}:=\theta_{n}+a_n^k$,  $k\geqslant 2.$ 
The measure $\mu$ is given by 
$$ \mu:= \sum_{n\geqslant 3}\sum_{k\geqslant 2}\frac{a_{n}}{k^2}\delta _{\zeta_{n,k}}. $$
To prove that $E$ is a uniqueness set for $\cD (\mu)$, we use Theorem \ref{unicite}. Let   $I_n = (\theta _{n}-a_n, \theta _n +a_n)$. By lemma \ref {lemcap}; we have $c_\mu(\zeta_n ) \geq \sqrt{a_n} $. Then 
\begin{eqnarray*}
\displaystyle \sum _n|I_n|\log \left ( \frac{c_\mu (I_n\cap E)}{|I_n|} \right)  &=&  \displaystyle \sum _n|I_n|\log \left ( \frac{c_\mu (\zeta _n)}{|I_n|} \right) \\
& \geq & \frac{1}{2}\displaystyle \sum _n a_n\log \left ( \frac{1}{a_n }\right)=+\infty 
\end{eqnarray*}
and the proof is complete.
\end{proof}

\subsection{} {Now we give  an example of positive measure  $\mu$ on $\TT$ without atoms and  a perfect  closed set $E$ such that   $E$ is uniqueness set for $\cD(\mu)$. Let $\ell_n$ be a positive sequence  and let $K$ be the associated generalized Cantor set.  Let $\alpha\in (0,1)$ and consider the measure $d\mu(\zeta)=\dist(\zeta,K)^\alpha dm(\zeta)$.  By \cite[Theorem 2]{EEK} and  \cite[Theorem 2]{EL}  
$$c_\mu(K)=0\iff \sum_{n}2^{-n}\ell_n^{-\alpha}=\infty.$$
Let $I_n=(e^{i(\log(n+1))^{-1}},e^{i(\log n)^{-1}})$. Note that 
$$\sum_{n\geq 2}  |I_n|\log |I_n|=-\infty.$$  We reproduce the generalized Cantor set in each $I_n$, denoted by $K_n$. We have $$c_\mu (K_n\cap I_n)\asymp c_\mu (K) |I_n|^\alpha.$$
Consider now $E=\{1\}\cup_n K_n$, we have $c_\mu(E)=0$ and we get 
$$\sum_{n}|I_n|\log \frac{|I_n|}{c_\mu(E\cap I_n)}=-\infty.$$
So by Theorem \ref{unicite},  $E$ is uniqueness set for $\cD(\mu)$. 
}

\end{document}